\documentclass[oneside,reqno]{amsart}
\usepackage{amsmath}
\usepackage{enumerate}
\newtheorem{theorem}{Theorem}[section]

\theoremstyle{definition}
\newtheorem{definition}[theorem]{Definition}

\numberwithin{equation}{section}

\newcommand{\blankbox}[5]

\begin{document}
\title{Sharp Weak Bounds for $p$-adic Hardy operators on $p$-adic Linear Spaces}
\author{Amjad Hussain$^{1,*}$}
\author{Naqash Sarfraz$^1$}
\author{Ferit G\"{u}rb\"{u}z$^2$}
\footnote{$^{}$Department of Mathematics, Quaid-I-Azam University 45320, Islamabad 44000, Pakistan\\
$^*$Corresponding Author: ahabbasi123@yahoo.com \\
$^2$Hakkari University, Faculty of Education, Department of Mathematics Education, Hakkari 30000, Turkey}


\begin{abstract} In this article, we establish sharp weak endpoint estimates for $p$-adic fractional Hardy operator. In addition, sharp weak bounds for $p$-adic Hardy operator on $p$-adic central Morrey space are also obtained.
\end{abstract}

\maketitle

\section{\textbf{Introduction}}

Let $p$ be a prime number, if a non-zero rational number $x$ can be written in the form
$$x=\frac{s}{t}p^{k},$$ where the integer $k=k(x)\in \mathbb Z$  and $s,t \in\mathbb Z$ are not multiples of $p,$ then the function:
$$|\cdot|_{p}:\mathbb Q\setminus \{0\}\rightarrow \mathbb R,$$
defined as:
$$|0|_{p}=0,\quad |x|_{p}=p^{-k},$$
satisfies all the axioms of a field norm with an additional property that:
\begin{equation}\label{NE}|x+y|_p\le\max\{|x|_p,|y|_p\},\end{equation}
and is commonly known as $p$-adic norm.

Likewise, the completion of field of rational numbers  $\mathbb{Q}$ with respect to ultrametric $p$-adic norm $|\cdot|_{p}$ is denoted by $\mathbb Q_p$ and is known as the field of $p$-adic numbers.
In expanded form, a $p$-adic number $x\neq 0$ can be written as (see \cite{VVZ}):
 \begin{equation}\label{EE3} x=p^k(\alpha_{0}+\alpha_{1}p+\alpha_{1}p^{2}+... )\end{equation}
 where $\alpha_i,k\in\mathbb Z,\alpha_0\ne0,\alpha_{i}\in\{0,1,2,...,p-1\},  i=1,2,\cdot\cdot\cdot.$ The series (\ref{EE3}) is convergent with respect to $p$-adic norm, since $|p^k\alpha_{i}p^{i}|_{p}\le p^{-k-i}.$

The higher dimensional space $\mathbb Q_p^n$ represents the vector space over $\mathbb Q_p$ which consists of all points $\mathbf{x}=(x_{1},x_{2},...,x_{n}),$ where $x_i\in\mathbb Q_p,i=1,2,...,n.$  The $p$-adic norm on the very space $\mathbb Q_p^n$ is:
  \begin{equation*}|\mathbf{x}|_{p}=\max_{1\leq j \leq n}|x_{j}|_{p}.\end{equation*}

In non-Archimedean geometry, the ball and and its boundary are defined, respectively, as:
$$B_{k}(\mathbf{a})=\{\mathbf{x} \in \mathbb{Q}_p^n:|\mathbf{x}-\mathbf{a}|_{p}\leq p^{k}\}, \
S_{k}(\mathbf{a})=\{\mathbf{x} \in \mathbb{Q}_p^n:|\mathbf{x}-\mathbf{a}|_{p}=p^{k}\}.$$
We denote $B_k(\mathbf{0})=B_k$ and $\ S_k(\mathbf{0})=S_k$ for convenience. Additionally, for every $\mathbf{a}_0\in\mathbb Q_p^n, \ \mathbf{a}_0+B_k=B_k(\mathbf{a}_0)$ and $\mathbf{a}_0+S_k=S_k(\mathbf{a}_0)$.

The local compactness and commutativity of the group $\mathbb{Q}_p^n$ under addition implies the existence of Haar measure $d\mathbf{x}$ on $\mathbb{Q}_p^n$, such that
$$\int_{B_{0}}d\mathbf{x}=|B_{0}|=1,$$
where the notation $|B|$ refers to the Haar measure of a measurable subset $B$ of $\mathbb{Q}_p^n.$ Also, it is not difficult to show that $|B_{k}(\mathbf a)|=p^{nk}$, $|S_{k}(\mathbf a)|=p^{nk}(1-p^{-n}),$ for any $\mathbf{a}\in\mathbb{Q}_p^n.$

Let $w(\mathbf{x})$ be a nonnegative locally integrable function on $\mathbb{Q}_p^n$ and and $w(E)$ the weighted measure of measurable subset $E\subset\mathbb{Q}_p^n,$ that is $w(E)=\int_Ew(x)dx$ respectively. The space of all complex-valued functions $f$ with norm conditions:
$$\|f\|_{L^{q}(w;\mathbb{Q}_{p}^{n})}=\bigg(\int_{\mathbb{Q}_{p}^{n}}|f(\mathbf{x})|^{p}w(\mathbf{x})d\mathbf{x}\bigg)^{1/q}<\infty,$$
is denoted by $L^q(w,\mathbb{Q}_p^n),(0<q<\infty),$ and is known as weighted Lebesgue space. Note that $L^{q}(1,\mathbb{Q}_{p}^{n})=L^{q}(\mathbb{Q}_{p}^{n}).$

We define the weighted $p$-adic weak Lebesgue space $L^{q,\infty}(w;\mathbb{Q}_{p}^{n})$ as the set of all measurable functions $f$ satisfying:
$$\|f\|_{L^{q,\infty}(w,\mathbb{Q}_{p}^{n})}=\sup_{\lambda>0}\lambda w\bigg(\{\mathbf{x}\in\mathbb{Q}_{p}^{n}:|f(\mathbf{x})|>\lambda\}\bigg)^{1/q}<\infty.$$
When $w=1,$ we get the weak Lebesgue space $L^{q,\infty}(\mathbb{Q}_{p}^{n})$ defined in  \cite{S1}. Next, we give the definition of strong and weak $p$-adic central Morrey spaces, respectively.
\begin{definition}\cite{WMF} Let $1\leq q<\infty$ and $-1/q\leq\lambda<0.$ A function $f\in L^{p}_{loc}(\mathbb{Q}_p^n)$ is said to belong to central Morrey space if:
\begin{eqnarray*}\begin{aligned}\|f\|_{\dot{B}^{q,\lambda}(\mathbb{Q}_p^n)}=\sup_{\gamma\in\mathbb{Z}}\bigg(\frac{1}{|B_{\gamma}|_{H}^{{1+\lambda q}}}\int_{B_{\gamma}}|f(\mathbf{x})|^{q}d\mathbf{x}\bigg)^{1/q}<\infty.
\end{aligned}\end{eqnarray*}
When $\lambda=-1/q,$ then $\dot{B}^{q,\lambda}(\mathbb{Q}_p^n)=L^{q}(\mathbb{Q}_p^n).$ It is not hard to see that $\dot{B}^{q,\lambda}(\mathbb{Q}_p^n)$ is reduced to \{0\} whenever $\lambda<-1/q.$
\end{definition}
\begin{definition}\cite{WF} Let $1\leq q<\infty$ and $-1/q\leq\lambda<0.$ The $p$-adic weak central Morrey space $W\dot{B}^{q,\lambda}(\mathbb{Q}_p^n)$ is defined as
\begin{eqnarray*}\begin{aligned}W\dot{B}^{q,\lambda}(\mathbb{Q}_p^n)=\{f:\|f\|_{W\dot{B}^{q,\lambda}(\mathbb{Q}_p^n)}<\infty\},
\end{aligned}\end{eqnarray*}
where \begin{eqnarray*}\begin{aligned}\|f\|_{W\dot{B}^{q,\lambda}(\mathbb{Q}_p^n)}=\sup_{\gamma\in\mathbb{Z}}|B_{\gamma}|_{H}^{-\lambda-1/q}\|f\|_{WL^{q}(B_{\gamma})},\end{aligned}\end{eqnarray*}and $\|f\|_{WL^{q}(B_{\gamma})}$ is the local $p$-adic $L^{q}$-norm of $f(x)$ restricted to the ball $B_{\gamma}$, that is
\begin{eqnarray*}\begin{aligned}\|f\|_{WL^{q}(B_{\gamma})}=\sup_{\lambda>0}|\{\mathbf{x}\in B_{\gamma}:|f(\mathbf{x})|>\lambda\}|^{1/q}.
\end{aligned}\end{eqnarray*}
It is clear that if $\lambda=-1/q,$ then  $W\dot{B}^{q,\lambda}(\mathbb{Q}_p^n)=L^{q,\infty}(\mathbb{Q}_p^n).$ Also, $\dot{B}^{q,\lambda}(\mathbb{Q}_p^n)\subseteq W\dot{B}^{q,\lambda}(\mathbb{Q}_p^n)$ for $1\leq q<\infty$ and $-1/q<\lambda<0.$
\end{definition}

The study of $p$-adic analysis is considered important in modern ages, because, $p$-adic analysis is a natural base for development of various models of ultrametric diffusion energy landscape \cite{ABKO}. It also attracts great deal of interest towards quantum mechanics \cite{VVZ}, theoretical biology \cite{DGSK}, string theory \cite{VV}, quantum gravity \cite{ADFV,BF}, spin glass theory \cite{ABK,PS1}. In \cite{ABKO}, it was shown that the $p$-adic analysis can be efficiently applied both to relaxation in complex speed systems and processes combined with the relaxation of a complex environment. Besides, $p$-adic analysis has an imperative role in pseudo-differential equations and harmonic analysis, see for example \cite{BV,CEKMM,CH,HS}.

For a non-negative integrable function $f$ on $\mathbb{R}^{+},$ the one dimensional Hardy operator is defined as:
\begin{equation}\label{EN4}h f(x)=\frac{1}{x}\int_{0}^{x}f(y)dy,\quad   x>0,\end{equation}
Hardy, in \cite{H}, established the following integral inequality for the Hardy operator:
\begin{equation}\label{EN5}\|h f\|_{L^{q}(\mathbb{R}^{+})}\leq\frac{q}{q-1}\|f\|_{L^{q}(\mathbb{R}^{+})},\quad 1<q<\infty.\end{equation}
 Later, Faris in \cite{F} gave the following extension of one dimensional Hardy operator:
\begin{equation}\label{EN6}\mathcal{H}f(\mathbf{x})=\frac{1}{|\mathbf{x}|^{n}}\int_{|\mathbf{y}|\leq|\mathbf{x}|}f(\mathbf{y})d\mathbf{y}.\end{equation}
 In \cite{CG}, the authors computed the norm of Hardy operator on the Lebesgue space, whereas Fu et al., in \cite{FGLZ}, obtained the boundedness of $\mathcal{H}$ on power weighted Lebesgue space $L^{q}(|\mathbf{x}|^{\alpha})$ for $1<q<\infty.$ For more details about the boundedness of Hardy operator we included some references \cite{B,H1,PS,ZFL}.

On the other hand, the fractional Hardy operator is obtained by merely interchanging $|\cdot|^n$ with $|\cdot|^{n+\alpha}$ ($0\leq\alpha<n$) in  (\ref{EN6}). The weak and strong type optimal bounds for the fractional Hardy and adjoint Hardy operator has also spotlighted many researchers in the past, see for example \cite{FGLZ,GZ1,GHZ,HJ,ZL}.

 For $f\in L_{\rm{loc}}(\mathbb{Q}_{p}^{n})$ and $0\le\alpha<n,$ Wu, in (\cite{W}), defined the $p$-adic fractional Hardy operator as:
$$H^{p}_{\alpha}f(\mathbf{x})=\frac{1}{|\mathbf{x}|_{p}^{n-\alpha}}\int_{|\mathbf{y}|_{p}\leq |\mathbf{x}|_{p}}f(\mathbf{y})d\mathbf{y},  \qquad\mathbf{x}\in\mathbb{Q}_{p}^{n}\setminus\{\mathbf{0}\}.$$
If $\alpha=0,$ the fractional $p$-adic Hardy operator is reduced to $p$-adic Hardy operators, see \cite{FWL}. Some other papers showing the boundedness of $p$-adic Hardy type operator include \cite{GZ,HS1,LZ,WMF}.

In this article, we give the sharp weak endpoint estimates for $p$-adic fractional Hardy operator on $p$-adic Lebesgue space. Moreover, Sharp weak bounds for $p$-adic Hardy operator on $p$-adic central morrey space are also obtained.

\section{\textbf{Sharp weak endpoint estimates for $p$-adic fractional Hardy Operator}}

Our main result for this section read as:
\begin{theorem}\label{T2} Let $0<\alpha<n$ and $n+\gamma>0.$ If $f\in L^{1}(\mathbb Q_p^n),$ then
 \begin{eqnarray*}\begin{aligned}\|H^{p}_{\alpha}f\|_{L^{(n+\gamma)/(n-\alpha),\infty}(|\mathbf{x}|_{p}^{\gamma};\mathbb Q_p^n)}
\le&C\|f\|_{L^{1}(\mathbb Q_p^n)},
\end{aligned}\end{eqnarray*}
where the constant $$C=\bigg(\frac{1-p^{-n}}{1-p^{-(n+\gamma)}}\bigg)^{(n-\alpha)/(n+\gamma)}$$ is optimal.
\end{theorem}

\begin{proof} Since
\begin{eqnarray}\begin{aligned}[b]\label{dd}|H^{p}_{\alpha}f(\mathbf{x})|
=&\bigg|\frac{1}{|\mathbf{x}|_{p}^{n-\alpha}}\int_{|\mathbf{y}|_{p}\leq|\mathbf{x}|_{p}}f(\mathbf{y})d\mathbf{y}\bigg|\\
=&\bigg|\frac{1}{|\mathbf{x}|_{p}^{n-\alpha}}\int_{|\mathbf{y}|_{p}\leq|\mathbf{x}|_{p}}f(\mathbf{y})d\mathbf{y}\bigg|
\\\leq&|\mathbf{x}|_{p}^{-(n-\alpha)}\|f\|_{L^{1}(\mathbb Q_p^n)}.
\end{aligned}\end{eqnarray}
Let $C_1 =\|f\|_{L^{1}(\mathbb Q_p^n)},$ then
\begin{eqnarray*}\begin{aligned}\{\mathbf{x}\in\mathbb{Q}_p^n:|H^{p}_{\alpha}f(\mathbf{x})|>\lambda\}\subset\{\mathbf{x}
\in\mathbb{Q}_p^n:|\mathbf{x}|_{p}<(C_1 /\lambda)^{1/(n-\alpha)}\}.
\end{aligned}\end{eqnarray*}
Thus,
\begin{eqnarray}\begin{aligned}[b]\label{NT1}&\|H^{p}_{\alpha}f\|_{L^{(n+\gamma)/(n-\alpha),\infty}(|x|_{p}^{\gamma};\mathbb Q_p^n)}\\
\le&\sup_{\lambda>0}\lambda\bigg(\int_{\mathbb{Q}_p^n}\chi_{\{\mathbf{x}
\in\mathbb{Q}_p^n:|H^{p}_{\alpha}f(\mathbf{x})|>\lambda\}}(\mathbf{x})|\mathbf{x}|_{p}^{\gamma}d\mathbf{x}\bigg)^{(n-\alpha)/(n+\gamma)}\\
\leq&\sup_{\lambda>0}\lambda\bigg(\int_{\mathbb{Q}_p^n}\chi_{\big\{\mathbf{x}
\in\mathbb{Q}_p^n:|\mathbf{x}|_{p}<\big(C_1 /\lambda\big)^{1/(n-\alpha)}\big\}}(\mathbf{x})|\mathbf{x}|_{p}^{\gamma}d\mathbf{x}\bigg)^{(n-\alpha)/(n+\gamma)}\\
=&\sup_{\lambda>0}\lambda\bigg(\int_{|\mathbf{x}|_{p}<\big(C_1 /\lambda\big)^{1/(n-\alpha)}}|\mathbf{x}|_{p}^{\gamma}d\mathbf{x}\bigg)^{(n-\alpha)/(n+\gamma)}\\
=&\sup_{\lambda>0}\lambda\bigg(\sum_{j=-\infty}^{\log_{p}\big(C_1 /\lambda\big)^{1/(n-\alpha)}}\int_{S_{j}}|\mathbf{x}|_{p}^{\gamma}d\mathbf{x}\bigg)^{(n-\alpha)/(n+\gamma)}\\
=&(1-p^{-n})^{(n-\alpha)/(n+\gamma)}\sup_{\lambda>0}\lambda
\bigg(\sum_{j=-\infty}^{\log_{p}\big(C_1 /\lambda\big)^{1/(n-\alpha)}}p^{j(n+\gamma)}d\mathbf{x}\bigg)^{(n-\alpha)/(n+\gamma)}\\
=&\bigg(\frac{1-p^{-n}}{1-p^{-(n+\gamma)}}\bigg)^{(n-\alpha)/(n+\gamma)}\sup_{\lambda>0}\lambda
\bigg(\frac{C_1 }{\lambda}\bigg)\\
=&\bigg(\frac{1-p^{-n}}{1-p^{-(n+\gamma)}}\bigg)^{(n-\alpha)/(n+\gamma)}\|f\|_{L^{1}(|\mathbf x|_p^\beta)}.
\end{aligned}\end{eqnarray}

To show that the constant
$$\bigg(\frac{1-p^{-n}}{1-p^{-(n+\gamma)}}\bigg)^{(n-\alpha)/(n+\gamma)},$$
appeared in (\ref{NT1}) is optimal, we employ the idea of use of
power function given in \cite{X}, hence, we let
\begin{eqnarray*}\begin{aligned}f_{0}(\mathbf{x})=\chi_{\{\mathbf{x}\in\mathbb Q_p^n:|\mathbf{x}|_{p}\leq1\}}(\mathbf{x}),
\end{aligned}\end{eqnarray*}
then
\begin{eqnarray*}\begin{aligned}[b]\|f_{0}\|_{L^{1}(\mathbb{Q}_p^n)}=1.
\end{aligned}\end{eqnarray*}
Also,
\begin{eqnarray*}\begin{aligned}H_\alpha^{p}f_{0}(\mathbf{x})=&\frac{1}{|\mathbf{x}|_{p}^{n-\alpha}}\int_{|\mathbf{y}|_{p}\le|\mathbf{x}|_{p}}f_{0}(\mathbf{y})d\mathbf{y}\\
=&\frac{1}{|\mathbf{x}|_{p}^{n-\alpha}}\int_{|\mathbf{y}|_{p}\le|\mathbf{x}|_{p}}\chi_{\{\mathbf{x}\in\mathbb Q_p^n:|\mathbf{y}|_{p}\leq1\}}(\mathbf{y})d\mathbf{y}\\
=& \frac{1}{|\mathbf{x}|_{p}^{n-\alpha}}\begin{cases}
      \int_{|\mathbf{y}|_{p}\le|\mathbf{x}|_{p}} d\mathbf{y},& |\mathbf{x}|_{p}\leq 1; \\
      \int_{|\mathbf{y}|_{p}\le1} d\mathbf{y}, & |\mathbf{x}|_{p}> 1.
   \end{cases}
\end{aligned}\end{eqnarray*}
Since $|B_{\log_p |\mathbf x|_p}|_H=|\mathbf x|_p^n|B_0|_H,$ therefore,
\begin{eqnarray*}\begin{aligned}H^{p}_{\alpha}f_{0}(\mathbf{x})= \begin{cases}
      |\mathbf{x}|_{p}^{\alpha},& |\mathbf{x}|_{p}\leq 1; \\
      |\mathbf{x}|_{p}^{\alpha-n}, & |\mathbf{x}|_{p}> 1.
   \end{cases}
\end{aligned}\end{eqnarray*}

Now,
\begin{eqnarray*}\begin{aligned}\{\mathbf{x}\in\mathbb{Q}_p^n:|H^{p}_{\alpha}f_{0}(\mathbf{x})|>\lambda\}=&\{|\mathbf{x}|_{p}\leq1:|\mathbf{x}|_{p}^{\alpha} >\lambda\}\cup\{|\mathbf{x}|_{p}>1:|\mathbf{x}|_{p}^{\alpha-n}>\lambda\}.
\end{aligned}\end{eqnarray*}

Since $0<\alpha<n,$ therefore, when $\lambda\geq1,$ then
\begin{eqnarray*}\begin{aligned}\{\mathbf{x}\in\mathbb{Q}_p^n:|H^{p}_{\alpha}f_{0}(\mathbf{x})|>\lambda\}=\emptyset,
\end{aligned}\end{eqnarray*}
and when $0<\lambda<1,$ then
\begin{eqnarray*}\begin{aligned}\{\mathbf{x}\in\mathbb{Q}_p^n:|H^{p}_{\alpha}f_{0}(\mathbf{x})|>\lambda\}=\{\mathbf{x}\in\mathbb{Q}_p^n:(\lambda)^{1/n}<|\mathbf{x}|_{p}<(1/\lambda)^{1/n-\alpha}\}.
\end{aligned}\end{eqnarray*}
Ultimately we are down to:
\begin{eqnarray}\begin{aligned}[b]\label{n1}&\|\mathcal{H}^{p}_{\alpha}f_{0}\|_{L^{(n+\gamma)/(n-\alpha))),\infty}(|\mathbf{x}|_{p}^{\gamma};\mathbb Q_p^n)}\\=&\sup_{0<\lambda<1}\lambda
\bigg(\int_{\mathbb{Q}_p^n}\chi_{\{\mathbf{x}\in\mathbb{Q}_p^n:(\lambda)^{1/\alpha}<|\mathbf{x}|_{p}<(1/\lambda)^{1/(n-\alpha)}\}}
(\mathbf{x})|\mathbf{x}|_{p}^{\gamma}d\mathbf{x}\bigg)^{(n-\alpha)/(n+\gamma)}\\
=&\sup_{0<\lambda<1}\lambda\bigg(\int_{(\lambda)^{1/\alpha}<|\mathbf{x}|_{p}<(1/\lambda)^{1/(n-\alpha)}}|\mathbf{x}|_{p}^{\gamma}d\mathbf{x}\bigg)^{(n-\alpha)/(n+\gamma)}\\
=&(1-p^{-n})^{(n-\alpha)/(n+\gamma)}\sup_{0<\lambda<1}\lambda\bigg(\sum_{j=\log_{p}\lambda^{1/\alpha}+1}^{\log_{p}\lambda^{1/(\alpha-n)}}p^{j(n+\gamma)}\bigg)^{(n-\alpha)/(n+\gamma)}\\
=&(1-p^{-n})^{(n-\alpha)/(n+\gamma)}\sup_{0<\lambda<1}\lambda\bigg(\frac{p^{(\log_{p}\lambda^{1/\alpha}+1)
(n+\gamma)}-p^{(\log_{p}\lambda^{1/(\alpha-n)}+1)(n+\gamma)}}{1-p^{(n+\gamma)}}\bigg)^{(n-\alpha)/(n+\gamma)}\\
=&(1-p^{-n})^{(n-\alpha)/(n+\gamma)}\sup_{0<\lambda<1}\lambda\bigg(\frac{{\lambda^{(n+\gamma)/\alpha}}
-{\lambda^{(n+\gamma)/(\alpha-n)}}}{p^{-(n+\gamma)}-1}\bigg)^{(n-\alpha)/(n+\gamma)}\\
=&(1-p^{-n})^{(n-\alpha)/(n+\gamma)}\sup_{0<\lambda<1}\bigg(\frac{1-{\lambda^{(n+\gamma)/\alpha}}{\lambda^{(n+\gamma)/(n-\alpha)}}}{1-p^{-(n+\gamma)}}\bigg)^{(n-\alpha)/(n+\gamma)}\\
=&\bigg(\frac{1-p^{-n}}{1-p^{-(n+\gamma)}}\bigg)^{(n-\alpha)/(n+\gamma)}\sup_{0<\lambda<1}
\bigg(1-{\lambda^{(n+\gamma)/\alpha}}{\lambda^{(n+\gamma)/(n-\alpha)}}\bigg)^{(n-\alpha)/(n+\gamma)}\\
=&\bigg(\frac{1-p^{-n}}{1-p^{-(n+\gamma)}}\bigg)^{(n-\alpha)/(n+\gamma)}\\
=&\bigg(\frac{1-p^{-n}}{1-p^{-(n+\gamma)}}\bigg)^{(n-\alpha)/(n+\gamma)}\|f_0\|_{L^{1}(\mathbb Q_p^n)}.
\end{aligned}\end{eqnarray}
We thus conclude from (\ref{NT1}) and (\ref{n1}) that
\begin{eqnarray*}\begin{aligned}\|{H}^{p}_{\alpha}\|_{{L^{1}(\mathbb Q_p^n)\rightarrow L^{(n+\gamma)/(n-\alpha),\infty}(|\mathbf{x}|_{p}^{\gamma};\mathbb Q_p^n)}}
=\bigg(\frac{1-p^{-n}}{1-p^{-(n+\gamma)}}\bigg)^{1/q}.
\end{aligned}\end{eqnarray*}
\end{proof}

\section{\textbf{Sharp bound for $p$-adic Hardy operator}}
The present section investigates the boundedness of $p$-adic Hardy operator on $p$-adic weak central Morrey spaces. It is shown the constant obtained in this case is also optimal.
\begin{theorem}\label{T1} Let $1\leq q<\infty,$ $-1/q\leq\lambda<0$ and if $f\in \dot{B}^{q,\lambda}(\mathbb{Q}_p^n),$ then
\begin{eqnarray*}\begin{aligned}\|Hf\|_{W\dot{B}^{q,\lambda}(\mathbb{Q}_p^n)}\le\|f\|_{\dot{B}^{q,\lambda}(\mathbb{Q}_p^n)},
\end{aligned}\end{eqnarray*}
and the constant $1$ is optimal.
\end{theorem}
\begin{proof}Applying H\"{o}lder's inequality:
\begin{eqnarray*}\begin{aligned}|H^{p}f(\mathbf{x})|\leq&\frac{1}{|\mathbf{x}|_{p}^{n}}\bigg(\int_{B(0,|\mathbf{x}|_{p})}|f(\mathbf{y})|^{q}d\mathbf{y}\bigg)^{1/q}
\bigg(\int_{B(0,|\mathbf{x}|_{p})}d\mathbf{y}\bigg)^{1/q'}\\
=&|\mathbf{x}|_{p}^{n\lambda}\|f\|_{\dot{B}^{q,\lambda}(\mathbb{Q}_p^n)}.
\end{aligned}\end{eqnarray*}
Let $C_2=\|f\|_{\dot{B}^{q,\lambda}(\mathbb{Q}_p^n)}.$ Since $\lambda<0,$ then
\begin{eqnarray*}\begin{aligned}\|Hf\|_{W\dot{B}^{q,\lambda}(\mathbb{Q}_p^n)}\leq&\sup_{\gamma\in\mathbb{Z}}\sup_{t>0}t|B_{\gamma}|_{H}^{-\lambda-1/q}\big|\{\mathbf{x}\in B_{\gamma}:C_2|\mathbf{x}|_{p}^{n\lambda}>t\}\big|^{1/q}\\
=&\sup_{\gamma\in\mathbb{Z}}\sup_{t>0}t |B_{\gamma}|_{H}^{-\lambda-1/q}\big|\{|\mathbf{x}|_{p}\leq p^{\gamma}: |\mathbf{x}|_{p}<(t/C_2)^{1/n\lambda}\}\big|^{1/q}.
\end{aligned}\end{eqnarray*}
If $\gamma\leq\log_{p}(t/C_2)^{1/n\lambda},$ then for $\lambda<0,$ we obtain
\begin{eqnarray*}\begin{aligned}&\sup_{t>0}\sup_{\gamma\leq\log_{p}(t/C_2)^{1/n\lambda}}t |B_{\gamma}|_{H}^{-\lambda-1/q}\big|\{|\mathbf{x}|_{p}\leq p^{\gamma}: |\mathbf{x}|_{p}<(t/C_2)^{1/n\lambda}\}\big|^{1/q}\\
&\leq\sup_{t>0}\sup_{\gamma\leq\log_{p}(t/C_2)^{1/n\lambda}}t p^{-\gamma n\lambda}\\&=C_2\\&=\|f\|_{\dot{B}^{q,\lambda}(\mathbb{Q}_p^n)}.
\end{aligned}\end{eqnarray*}
If $\gamma>\log_{p}(t/C_2)^{1/n\lambda},$ then for $\lambda+1/q>0,$ we get
\begin{eqnarray*}\begin{aligned}&\sup_{t>0}\sup_{\gamma>\log_{p}(t/C_2)^{1/n\lambda}}t |B_{\gamma}|_{H}^{-\lambda-1/q}|\{|\mathbf{x}|_{p}\leq p^{\gamma}: |\mathbf{x}|_{p}<(t/C_2)^{1/n\lambda}\}|^{1/q}\\
&\leq\sup_{t>0}\sup_{\gamma>\log_{p}(t/C_2)^{1/n\lambda}}t p^{-\gamma n(\lambda+1/q)}(t/C_2)^{1/q\lambda}\\&=C_2\\&=\|f\|_{\dot{B}^{q,\lambda}(\mathbb{Q}_p^n)}.
\end{aligned}\end{eqnarray*}
Therefore,
\begin{eqnarray}\begin{aligned}\label{1}\|Hf\|_{W\dot{B}^{q,\lambda}(\mathbb{Q}_p^n)}\leq\|f\|_{\dot{B}^{q,\lambda}(\mathbb{Q}_p^n)}.
\end{aligned}\end{eqnarray}

Conversely, to prove that the constant $1$ is optimal, we let
 $$f_{0}(\mathbf{x})=\chi_{\{|\mathbf{x}|_{p}\leq1\}}(\mathbf{x}),$$\\
then, $$\|f_0\|_{\dot{B}^{q,\lambda}(\mathbb{Q}_p^n)}=\sup_{\gamma\in\mathbb{Z}}\bigg(\frac{1}{|B_{\gamma}|_{H}^{{1+\lambda q}}}\int_{B_{\gamma}}\chi_{\{|\mathbf{x}|_{p}\leq1\}}(\mathbf{x})d\mathbf{x}\bigg)^{1/q}.$$
If $\gamma<0,$ then
\begin{eqnarray*}\begin{aligned}\sup_{\substack{\gamma\in\mathbb{Z} \\ \gamma<0}}\bigg(\frac{1}{|B_{\gamma}|_{H}^{{1+\lambda q}}}\int_{B_{\gamma}}d\mathbf{x}\bigg)^{1/q}
=\sup_{\substack{\gamma\in\mathbb{Z} \\ \gamma<0}}p^{-n\gamma\lambda }=1,\end{aligned}\end{eqnarray*}
since $\lambda<0.$
If $\gamma\ge0,$ then using the condition that $\lambda+1/q>0,$ we have
\begin{eqnarray*}\begin{aligned}\sup_{\substack{\gamma\in\mathbb{Z} \\ \gamma\ge0}}\bigg(\frac{1}{|B_{\gamma}|_{H}^{{1+\lambda q}}}\int_{B_0}d\mathbf{x}\bigg)^{1/q}=\sup_{\substack{\gamma\in\mathbb{Z} \\ \gamma\ge0}}p^{-n\gamma(\lambda+1/q) }=1.\end{aligned}\end{eqnarray*}
Therefore,
$$\|f_0\|_{\dot{B}^{q,\lambda}(\mathbb{Q}_p^n)}=1.$$

Moreover,\begin{eqnarray*}\begin{aligned}H^{p}f_{0}(\mathbf{x})= \begin{cases}
      1,& |\mathbf{x}|_{p}\leq 1; \\
      |\mathbf{x}|_{p}^{-n}, & |\mathbf{x}|_{p}> 1,
   \end{cases}
\end{aligned}\end{eqnarray*}
which implies that $|H^pf_0(\mathbf x)|\le1.$

Next, in order to construct weak central Morrey norm we divide our analysis into following two cases:\\
Case 1. When $\gamma\leq0,$ then
\begin{eqnarray*}\begin{aligned}\|Hf_{0}\|_{WL^{q}(B_{\gamma})}=\sup_{0<t\leq1}t|\{\mathbf{x}\in B_{\gamma}:1>t\}|^{1/q}=p^{n\gamma/q},
\end{aligned}\end{eqnarray*}
and
\begin{eqnarray*}\begin{aligned}\|Hf_{0}\|_{W\dot{B}^{q,\lambda}(\mathbb{Q}_p^n)}=\sup_{\gamma<0}|B_{\gamma}|_{H}^{-\lambda-1/q}\|f\|_{WL^{q}(B_{\gamma})}
=\sup_{\gamma<0}p^{-n\gamma\lambda}=1=\|f_0\|_{\dot{B}^{q,\lambda}(\mathbb{Q}_p^n)}.
\end{aligned}\end{eqnarray*}
\\
Case 2. When $\gamma>0,$ we have
\begin{eqnarray*}\begin{aligned}\|Hf_{0}\|_{WL^{q}(B_{\gamma})}=&\sup_{0<t\leq1}t|\{\mathbf{x}\in B_{\mathbf{0}}:1>t\}\cup\{1\leq|\mathbf{x}|_{p}<p^{\gamma}:|\mathbf{x}|_{p}^{-n}>t\}|^{1/q}.
\end{aligned}\end{eqnarray*}
For further analysis, this case is further divided into the following subcases:\\
Case 2(a). If $1<\gamma<\log_{p}t^{-1/n},$ then
\begin{eqnarray*}\begin{aligned}\|Hf_{0}\|_{WL^{q}(B_{\gamma})}=\sup_{0<t\leq1}t\{1+p^{n\gamma}-1\}^{1/q}
=\sup_{0<t\leq1}tp^{n\gamma/q}.
\end{aligned}\end{eqnarray*}
Case 2(b). If $1<\log_{p}t^{-1/n}<\gamma,$ then:\begin{eqnarray*}\begin{aligned}\|Hf_{0}\|_{WL^{q}(B_{\gamma})}=&\sup_{0<t\leq1}t(1+t^{-1}-1)^{1/q}=&\sup_{0<t\leq1}t^{1-1/q}.
\end{aligned}\end{eqnarray*}
Now, for $1\leq q<\infty$ and $-1/q\leq\lambda<0,$ from case 2(a) and 2(b), we obtain
\begin{eqnarray}\begin{aligned}[b]\label{2}&\|Hf_{0}\|_{W\dot{B}^{q,\lambda}(\mathbb{Q}_p^n)}\\
&=\max\bigg\{\sup_{0<t\leq1}\sup_{1<\gamma\leq\log_{p}(1/t)^{1/n}}tp^{-n\gamma\lambda},\sup_{0<t\leq1}\sup_{1<\log_{p}(1/t)^{1/n}<\gamma} t^{1-1/q}p^{-n\gamma(\lambda+1/q)}\bigg\}\\
&=\max\bigg\{\sup_{0<t\leq1}t^{1+\lambda},\sup_{0<t\leq1}t^{1+\lambda}\bigg\}
\\&=1=\|f_0\|_{\dot{B}^{q,\lambda}(\mathbb{Q}_p^n)}.
\end{aligned}\end{eqnarray}
Finally, using (\ref{1}) and (\ref{2}), we arrive at:
\begin{eqnarray*}\begin{aligned}\|H\|_{\dot{B}^{q,\lambda}(\mathbb{Q}_p^n)\rightarrow W\dot{B}^{q,\lambda}(\mathbb{Q}_p^n)}=1.
\end{aligned}\end{eqnarray*}
\end{proof}

\noindent\textbf{Conflict of Interest:} The authors declare that they have no conflict of interest.\\

\noindent\textbf{Data Availability Statement:} No data were used to support this study.\\

\noindent\textbf{Funding Statement:} This research is partially supported by Higher Education Commission (HEC) NRPU Programme 2017-18 [7098/Federal/ NRPU/R\&D/HEC/ 2017] and the Quaid-I-Azam University Research Fund [URF-2019].


\begin{thebibliography}{99}
 \bibitem{ADFV}    I.Y.  Arefeva, B. Dragovich, P. Frampton and I.V. Volovich, The wave function of the universe and $p$-adic gravity, Mod. Phys. Lett. A, {\bf 6}  (1991), 4341--4358.
\bibitem{ABK}    V.A.  Avetisov, A.H. Bikulov and S.V. Kozyrev, Application of $p$-adic analysis to models of sponteneous breaking of replica symmetry, J. Phys. A: Math. Gen. {\bf 32(50)} (1999), 8785-8791.
\bibitem{ABKO}    V.A.  Avetisov, A.H. Bikulov, S.V. Kozyrev and V.A. Osipov, $p$-adic models of ultrametric diffusion constrained by hierarchical energy landscaapes, J. Phys. A:Math.Gen. {\bf 35} (2002), 177--189.
\bibitem{BV}    R.A. Bandaliyev, S.S. Volosivets, Hausdorff operator on weighted Lebesgue and grand Lebesgue p-adic spaces, $p$-adic Numb. Ultrametric Anal. Appl. {\bf11(2)} (2019), 114--122.
\bibitem{B}    A.G.  Bliss, An integral inequality, J. Lond. Math. Soc. {\bf 5} (1930), 40--46.
\bibitem{BF}    L. Brekke and P.G.O. Frued, $p$-adic numbers in Physics, Phys. Rep. {\bf 233}  (1993), 1--66.
\bibitem{CG}    M. Christ, L. Grafakos, Best Constants for two non convolution inequalities, Proc. Amer. Math. Soc., {\bf123} (1995), 1687--1693.
\bibitem{CEKMM}      N.M. Chuong, Y.V. Egorov, A.Y. Khrennihov, Y.Meyer and D. Mumford, Harmonic, Wavelet and $p$-adic Analysis, World Scientific, 2007.
\bibitem{CH}      N.M. Chuong and H.D. Hung, Maximal functions and weighted norm inequalities on local fields, Appl. Comput. Harmon. Anal.,{\bf29} (2010), 272--286.
\bibitem{DGSK}      D. Dubischar, V.M. Gundlach, O. Steinkamp and A. Khrennikov, A $p$-adic model for the process of thinking disturbed by physiological and information noise, J. Theor. Biol. {\bf197} (1999), 451--467.
\bibitem{F}     W.G. Faris, Weak Lebesgue spaces and quantum mechanical binding, Duke Math. J. {\bf{43}} (1976), 365--373.
\bibitem{FGLZ}  Z.W. Fu, L. Grafakos, S.Z. Lu and F.Y. Zhao, Sharp bounds for $m$-linear Hardy and Hilbert Operators, Houston. J. Math. {\bf 38(1)} (2012), 225--244.
\bibitem{FWL}   Z.W. Fu, Q.Y. Wu and S.Z. Lu, Sharp estimates of $p$-adic Hardy and Hardy-Littlewood-P\'{o}lya Operators, Acta Math. Sinica {\bf{29}} (2013) 137--150.
\bibitem{GZ1}   G. Gao and F.Y. Zhao, Sharp weak bounds for Hausdorff operators, Anal Math, {\bf{41(3)}} (2015), 163--173.
\bibitem{GHZ}   G. Gao, X. Hu and C. Zhong, Sharp weak estimates for Hardy-type Operators, Ann. Funct. Anal. {\bf 7(3)} (2016), 421--433.
\bibitem{GZ}    G. Gao and Y. Zhong, Some estimates of Hardy Operators and their commutators on Morrey-Herz spaces, J. Math. Inequal. {\bf 11(1)} (2017), 49--58.
\bibitem{H}     G.H. Hardy, Note on a theorem of Hilbert, Math. Z., {\bf 6} (1920), 314--317.
\bibitem{H1}    K.P. Ho, Hardy's inequality on Hardy–Morrey spaces, Georg. Math. J., {\bf26(3)} (2019), 405--413.
\bibitem{HS}    A. Hussain , N. Sarfraz, The Hausdorff operator on weighted  $p$-adic Morrey and Herz type spaces, $p$-adic Numb. Ultrametric Anal. Appl. {\bf11(2)} (2019), 151--162.
\bibitem{HS1}   A. Hussain , N. Sarfraz, Optimal weak type estimates for $p$-adic Hardy operators, $p$-adic Numb. Ultrametric Anal. Appl. {\bf12(1)} 2020, 12--21.
\bibitem{LZ}    R.H. Liu and J. Zhou, Sharp estimates for the $p$-adic Hardy type Operator on higher-dimensional product spaces, J. Inequal. Appl.{\bf{2017}} (2017) 13PP.
\bibitem{PS1}    G. Parisi and N. Sourlas, $p$-adic numbers and replica symmetry, Eur. J. Phys. B  {\bf 14} (2000), 535--542.
\bibitem{PS}    L.-E. Persson and S.G. Samko, A note on the best constants in some hardy inequalities, J. Math. Inequal. {\bf 9 (2)} (2015), 437–-447.
\bibitem{VVZ}   V.S. Vladimirov, I.V. Volovich and E.I.Zelenov, $p$-adic Analysis and Mathematical Physics, World Scientific, Singapore, 1994.
\bibitem{VV}   V.S. Vladimirov and I.V. Volovich, $p$-adic quantum mechanics, Commun. Math. Phy.,{\bf 123} (1989), 659--676.
\bibitem{S1}     S.S. Volosivets, Weak and strong estimates for rough Hausdorff type operator defined on $p$-adic linear space, $p$-Adic Numb. Ultrametric Anal. Appl. {\bf9(3)} (2017), 222--230.
\bibitem{W}     Q.Y. Wu, Boundedness for Commutators of fractional $p$-adic Hardy Operator, J. Inequal. Appl. {\bf{2012}} (2012) 12pp.
\bibitem{WMF}   Q.Y. Wu, L. Mi and Z.W. Fu, Boundedness of $p$-adic Hardy Operators and their commutators on $p$-adic central Morrey and BMO spaces, J. Funct. Spaaces Appl.  {\bf{2013}} (2013), Art. ID 359193, 10pp.
\bibitem{WF}    Q.Y. Wu and Z.W. Fu, Hardy-Littlewood-Sobolev Inequalities on $p$-adic Central Morrey Spaces, Journal of Function Spaces Volume 2015, Article ID 419532, 7 pages
http://dx.doi.org/10.1155/2015/419532.
\bibitem{X}     J. Xiao, $L^p$ and $BMO$ bounds of weighted Hardy-Littlewood Averages, J. Math. Anal. Appl. {\bf262} (2001), 660--666.
\bibitem{HJ}    H. Yu and J. Li, Sharp weak estimates for $n$-dimensional fractional Hardy Operators, Front. Math. China {\bf 13(2)} (2018), 449--457.
\bibitem{ZL}    F.Y. Zhao and S.Z. Lu, The best bound for $n$-dimensional fractional Hardy Operator, Math. Inequal Appl, {\bf 18(1)} (2015), 233--240.
\bibitem{ZFL}   F.Y. Zhao, Z.W. Fu and S.Z. Lu, Endpoint estimates for n-dimensional Hardy operators and their commutators, Sci. China Math. {\bf 55(10)} (2012), 1977–-1990.

\end{thebibliography}
\end{document}